\documentclass{article}

\usepackage[margin=1.0in]{geometry}
\usepackage{amsmath}
\usepackage{amssymb}
\usepackage{amsthm}
\usepackage{mathabx}
\usepackage{enumerate}
\usepackage{graphicx}
\newtheorem{thm}{Theorem}[section]
\newtheorem{lem}[thm]{Lemma}
\newtheorem{dfn}{Definition}

\providecommand{\e}{\varepsilon}

\providecommand{\f}{\varphi}

\providecommand{\x}{\xi}

\providecommand{\R}{\mathbb{R}}

\title{Volumes spanned by $k$-point configurations in $\R^d$}
\author{Belmiro Galo and Alex McDonald}

\begin{document}

\maketitle

\begin{abstract}
Given a $k$-point configuration $x\in (\R^d)^k$, we consider the $\binom{k}{d}$-vector of volumes determined by choosing any $d$ points of $x$.  We prove that a compact set $E\subset \R^d$ determines a positive measure of such volume types if the Hausdorff dimension of $E$ is greater than $d-\frac{d-1}{2k-d}$.  This generalizes results of Greenleaf, Iosevich, and Mourgoglou \cite{GIM15}, Greenleaf, Iosevich, and Taylor \cite{GIT20}, and the second listed author \cite{M20}.
\end{abstract}

\section{Introduction}

A recurrent theme throughout mathematics is to show that if one has a set which is sufficiently structured in some way and applies a non-trivial map, the image is also structured.  A classic example of this theme is the Falconer distance problem, which is one of the most important and interesting problems in geometric measure theory.  Given a set $E\subset\R^d$, define its distance set to be

\[
\Delta(E)=\{|x-y|:x,y\in E\}.
\]

The Falconer distance problem asks how large the Hausdorff dimension of a compact set $E$ must be to ensure that $\Delta(E)$ has positive Lebesgue measure.  Falconer \cite{F86} proved that $\dim E>\frac{d+1}{2}$ implies $\Delta(E)$ has positive measure, where here and throughout $\dim E$ denotes the Hausdorff dimension of the set $E$.  He also found a family of examples $\{E_s\}$ such that for any $s<\frac{d}{2}$, one has $\dim E_s>s$ and $\mathcal{L}_1(\Delta(E))=0$.  This suggests what is now known as the Falconer distance problem, which asks for the smallest $s$ such that $\dim E>s$ implies $\mathcal{L}_1(\Delta(E))>0$.  Falconer's work implies this threshold is between $\frac{d}{2}$ and $\frac{d+1}{2}$, and it is conjectured that $\frac{d}{2}$ is in fact the correct threshold.  The first major results were due to Wolff \cite{W99} and Erdogan \cite{E06}, proving the threshold $\frac{d}{2}+\frac{1}{3}$ in the case $d=2$ and $d\geq 3$, respectively.  These were the best results until recently, when a number of improvements were made using the decoupling theorem of Bourgain and Demeter \cite{BD15}.  The best results currently state that for compact $E\subset \R^d$, the distance set $\Delta(E)$ has positive Lebesgue measure if $\dim E>s_d$ where $s_2=5/4$ \cite{GIOW20}, $s_3=9/5$ \cite{DGOWWZ18}, $s_d=\frac{d}{2}+\frac{1}{4}$ when $d\geq 4$ is even \cite{DIOWZ20}, and $s_d=\frac{d}{2}+\frac{1}{4}+\frac{1}{4(d-1)}$ when $d\geq 4$ is odd \cite{DZ19}.  \\

A key generalization of the Falconer distance problem comes from considering geometric properties of point configurations.  We first establish some notation.  We will use superscripts to denote vectors and subscripts to denote components of vectors, so for a configuration $x\in (\R^d)^k$ we have $x=(x^1,\cdots,x^k)$ where each $x^j\in\R^d$ has components $x^j=(x_1^j,\cdots,x_d^j)$.  The most direct generalization of the Falconer distance problem in this context is the problem of congruence classes of such configurations.  For $k\leq d$, the congruence class of $x\in(\R^d)^{k+1}$ is determined by the $\binom{k+1}{2}$-tuple of distances $|x^i-x^j|$.  Define $\Delta_k(E)$ to be the set of vectors $\{|x^i-x^j|\}_{1\leq i< j< k+1} \}$ with $x^i\in E$ for all $i$. Note that the set $\Delta_1(E)$ coincides with $\Delta(E)$ defined above.  Greenleaf, Iosevich, Liu, and Palsson \cite{GILP15} proved that $\Delta_k(E)$ has positive $\binom{k+1}{2}$ dimensional Lebesgue measure if $\dim E>d- \frac{d-1}{k+1}$.  The proof strategy was built on the fact that two configurations are congruent if and only if there is an isometry mapping one to the other, which allowed the authors to study the problem in terms of the group action.  The group action framework was instrumental in the proof of the discrete predecessor of the Falconer distance problem, known as the Erdos distinct distance problem, which asks for the minimum number of distances determined by a set of $N$ points in $\R^d$.  In that context the group action framework was introduced by Elekes and Sharir \cite{ES11} and ultimately used by Guth and Katz to resolve the problem in the plane, obtaining the bound $N/\log N$ which is optimal up to powers of $\log$ \cite{GK15}. \\

The configuration congruence problem becomes more subtle when $k>d$.  This is because the system of distance equations becomes overdetermined, and the space of congruence classes can no longer be identified with the space of distance vectors $\R^{\binom{k+1}{2}}$.  Invoking the group action framework again, one would expect heuristically that the space of congruence classes should have dimension $d(k+1)-\binom{d+1}{2}$, since the space of configurations has dimension $d(k+1)$ and the space of isometries has dimension $\binom{d+1}{2}$.  Chatziconstantinou, Iosevich, Mkrtchyan, and Pakianathan \cite{CIMP17} proved that in fact this heuristic is correct, and obtained a non-trivial dimensional threshold.  Their proof used the theory of combinatorial rigidity.  Given a $(k+1) $-point configuration, they proved that the congruence class was determined (up to finitely many choices) if one fixes $d(k+1)-\binom{d+1}{2}$ strategically chosen distances.  They then used the group action framework to prove that $\Delta_k(E)$ has positive $d(k+1)-\binom{d+1}{2}$ dimensional measure if $\dim E>d-\frac{1}{k+1}$. \\

The key to the results in \cite{GILP15} and \cite{CIMP17} is the fact that the congruence relation can be described in terms of action of the isometry group on the space of configurations.  It is therefore natural to study other point configuration problems where congruence is replaced by other geometric relations with a corresponding group action invariance.  One such problem occurs by considering the volumes which are obtained by choosing any $d$ points of a configuration.  More precisely, we make the following definition.

\begin{dfn}
The \textbf{volume type} of $x\in (\R^d)^k$ is the vector

\[
\{\det(x^{j_1},\cdots,x^{j_d})\}_{1\leq j_1<\cdots <j_d\leq k}\in \R^{\binom{k}{d}}.
\]

For a set $E\subset\R^d$, let 

\[
\mathcal{V}_{k,d}(E)=\{\{\det(x^{j_1},\cdots,x^{j_d})\}_{1\leq j_1<\cdots <j_d\leq k}:x^1,...,x^k\in E\}
\]

be the set of volume types determined by points in $E$.  Finally, let $\mathcal{V}_{k,d}=\mathcal{V}_{k,d}(\R^d)$ be the space of all volume types of $k$-point configurations in $\R^d$. 
\end{dfn}

Thus, the volume type of a $k$-point configuration $x\in(\R^d)^k$ encodes all volumes obtained by choosing any $d$ points from $x$ (see figure 1).

\begin{figure}[h!]
\centering
\includegraphics[scale=0.4]{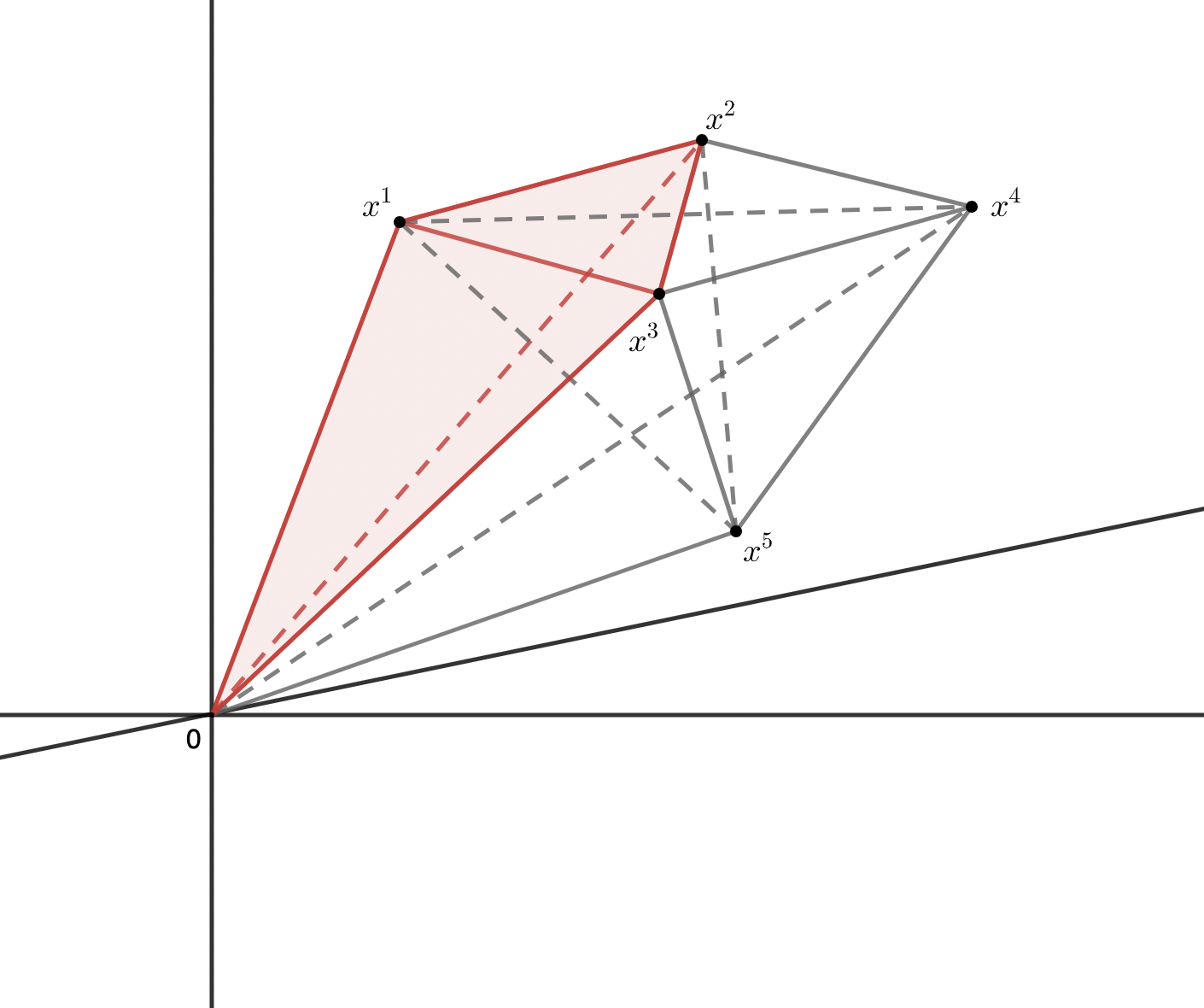}
\caption{5-point configuration $x\in \left(\R^3\right)^5$.}
\label{fig:universe}
\end{figure}

When $k=d$, the space of volume types is simply $\mathcal{V}_{d,d}=\R$, which we may equip with the Lebesgue measure.  In the case $k=d=3$, Greenleaf, Iosevich, and Mourgoglou \cite{GIM15} proved that $\mathcal{V}_{3,3}(E)$ has positive measure if $\dim E>13/5$.  This threshold was later improved and generalized to higher dimension by Greenleaf, Iosevich, and Taylor \cite{GIT20} who considered the case $k=d$ for any $d\geq 3$ and proved $\mathcal{V}_{d,d}(E)$ has positive measure if $\dim E>d-1+\frac{1}{d}$.  Notice in the case $d=3$ this improves the $13/5$ threshold to $7/3$.  When $k$ is large, the problem is overdetermined and hence one needs to define an appropriate measure on the space of volume types.  The second listed author \cite{M20} proved that $\mathcal{V}_{k+1,2}$ may be identified with a space of dimension $2k-1$ (note that this is consistent with our previously described heuristic, since the space of $(k+1)$-point configurations has dimension $2k+2$ and the Lie group $\text{SL}_2(\R)$ has dimension 3) and that $\mathcal{V}_{k+1,2}$ has positive measure if $\dim E>2-\frac{1}{2k}$.  The second author also obtained a non-trivial result in the two dimensional problem over finite fields and rings of the form $\mathbb{Z}/p^\ell\mathbb{Z}$ \cite{M19}. \\

Our first goal is to generalize these results to the case where $k,d$ are natural numbers satisfying $k\geq d\geq 2$ but are otherwise arbitrary.  Our heuristic suggests that the dimension of $\mathcal{V}_{k,d}$ should be $d(k-d)+1$.  Our first theorem shows that this is indeed the case.

\begin{thm}
\label{thm: manifold}
The set $\mathcal{V}_{k,d}$ is an embedded submanifold in $\R^{\binom{k}{d}}$ of dimension $d(k-d)+1$.
\end{thm}

This will be proved in Section 2.  It follows that $\mathcal{V}_{k,d}$ is equipped with $(d(k-d)+1)$-dimensional Lebesgue measure, which we will denote by $\mathcal{L}_{d(k-d)+1}$.  It also follows that if $E$ is compact, $\mathcal{V}_{k,d}(E)$ is a compact subset of $\mathcal{V}_{k,d}$.  \\

With this result, we are now ready to state our first main theorem.

\begin{thm}
\label{thm: MT1}
Let $k\geq d\geq 2$ and let $E\subset\R^d$ be a compact set with Hausdorff dimension greater than $d-\frac{d-1}{2k-d}$.  Then, $\mathcal{L}_{d(k-d)+1}(\mathcal{V}_{k,d}(E))>0$.
\end{thm}

We shall remark here that our decision to work with signed volume, rather than unsigned volume, is an arbitrary one.  One can immediately deduce an unsigned version of Theorem \ref{thm: MT1} by decomposing the set $\mathcal{V}_{k,d}(E)$ into $2^{\binom{k}{d}}$ pieces according to the sign of each component and applying the pigeonhole principle.  We also note that in the case $k=d$ our threshold is the same as the one in \cite{GIT20}.  The general case is proved by reducing to the $k=d$ case, so a better exponent in that case would yield better general results.  \\

Another classic object of study in the distance problem is chains of distances determined by a set.  A configuration $x\in(\R^d)^{k}$ determines a $k-1$ chain of distances $|x^1-x^2|,|x^2-x^3|,..., |x^{k-1}-x^k|$.  Bennett, Iosevich, and Taylor \cite{BIT16} proved that if $\dim E>\frac{d+1}{2}$ then the set of distance chains determined by $E$ has positive measure.  This result was later generalized by Iosevich and Taylor \cite{IT19} to apply to all trees.  \\

Our other main theorem will pertain to chains of volumes.  Since a volume is determined by $d$ points rather than $2$, we will consider chains in the sense of hypergraphs.  Recall an $r$-regular hypergraph is a set of vertices and hyperedges, where each hyperedge connects $r$ vertices (so, in particular, a $2$-regular hypergraph is just a graph).  A chain in a hypergraph is a seqeunce of vertices where each shares some hyperedge with the next.\\

Given a $d$-uniform hypergraph on vertices $\{1,...,k\}$ and a configuration $x\in (\R^d)^k$, we may consider volumes determined by points $x^{j_1},..., x^{j_d}$ such that $(j_1,...,j_d)$ forms a hyperedge.  In this framework, Theorem \ref{thm: MT1} gives a result in the case where the hypergraph is complete.  Our methods also allow us to obtain a result in the case of a chain.  This is our next theorem.

\begin{thm}
Let $E\subset\R^d$ be compact, and let 

\[
\mathcal{C}_{k,d}=\{\{\det(x^j,x^{j+1},\cdots,x^{j+d-1})\}_{1\leq j\leq k+1-d}:x^1,...,x^k\in E\}.
\]

If $\dim E>d-1+\frac{1}{d}$, then $\mathcal{L}_{(k+1-d)}(\mathcal{C}_{k,d}(E))$.

\label{thm: MT3}
\end{thm}

Here, we pause to make a couple remarks.  First, note that if our set $E$ is contained in a hyperplane through the origin it cannot determine any non-zero volume, so the optimal threshold cannot be smaller than $d-1$.  Second, it is interesting to note that the threshold in Theorem \ref{thm: MT3} does not depend on $k$ whereas the threshold in Theorem \ref{thm: MT1} tends to $d$ as $k\to\infty$.  Our final theorem shows that this cannot be avoided.

\begin{thm}[Sharpness]
\label{thm:sharpness}
For any $k\geq d\geq 2$ and any 

\[
s_{k,d}<d-\frac{d^2(d-1)}{d(k-1)+1},
\]

there exists a compact set $E_{k,d}\subset \R^d$ such that $\dim E_{k,d}>s_{k,d}$ and $\mathcal{L}_{d(k-d)+1}(\mathcal{V}_{k,d}(E_k)) =0$.
\end{thm}

\section{Setting up the group action framework}

We start by examining the relationship between volume types and the action of $\text{SL}_d(\R)$ on the space of configurations.  Generically, the property of two configurations having the same volume type is equivalent to those configurations lying in the same orbit of this action.  However, this equivalence breaks down for configurations which do not span $\R^d$.  This leads to the following definition.

\begin{dfn}
A configuration $x\in(\R^d)^k$ is called \textbf{degenerate} if $\{x^1,\cdots,x^d\}$ is linearly dependent, and \textbf{non-degenerate} otherwise.
\end{dfn}

We remark that we could broaden this notion of non-degeneracy to include configurations where any $d$ points span $\R^d$, not just the first $d$ points.  However, in either case the set of degenerate configurations are negligible so we have chosen this definition to simplify our proofs and notation. \\

With our definition in place, we have the following lemma.

\begin{lem}
\label{lem: GroupAction}
Let $x,y\in (\R^d)^k$ be non-degenerate.  Then $x$ and $y$ have the same volume type if and only if there exists a unique $g\in\text{SL}_d(\R)$ such that $y=gx$ (i.e., for each $j$ we have $y^j=gx^j$).
\end{lem}
\begin{proof}  First, suppose  $x$ and $y$ have the same volume types. Because $x$ and $y$ are non-degenerate, 

\begin{equation*}
D:=\det(x^{1},\cdots,x^{d})=\det(y^{1},\cdots,y^{d})\neq 0. 
\end{equation*}

 Equivalently, the $d\times d$ matrix with columns $x^{1} \cdots x^{d}$ is non-singular, same as $y^{1} \cdots y^{d}$. We denote these matrices by $(x^{1} \cdots x^{d})$ and $(y^{1} \cdots y^{d})$, respectively. Let
\begin{equation*}
    g=(y^{1} \cdots y^{d})(x^{1} \cdots x^{d})^{-1}.
\end{equation*}

This equation means that $(gx^{1} \cdots gx^{d})=(y^{1} \cdots y^{d})$, so $gx^{n}=y^{n}$ for every $1\leq n\leq d$.  Let $i$ be any index, and write
 \begin{equation*}
    x^{i}=\sum_{n=1}^{d}a_nx^{n},~~ y^{i}=\sum_{n=1}^{d}b_ny^{n}.
\end{equation*}

Observe that

\begin{equation*}
 \det(x^{1}, \cdots, x^{d-1}, x^{i})=\det\left(x^1,\cdots,x^{d-1}, \sum_{n=1}^{d}a_nx^{n}\right)=\sum_{n=1}^{d}\det\left(x^1,\cdots,x^{d-1}, a_nx^{n}\right)
\end{equation*}

since the determinant behaves like a linear function on the rows of the matrix. Therefore, 
\begin{equation*}
 \det(x^{1}, \cdots, x^{d-1}, x^{i})=\det\left(x^1,\cdots,x^{d-1}, a_dx^{d}\right)=a_dD.
\end{equation*}
The same conclusion holds for 

\begin{equation*}
 \det(y^{1}, \cdots, y^{d-1}, y^{i})=\det\left(y^1,\cdots,y^{d-1}, \sum_{n=1}^{d}b_ny^{n}\right)=\sum_{n=1}^{d}\det\left(y^1,\cdots,y^{d-1}, b_ny^{n}\right)=b_dD.
\end{equation*}

By assumption $x$ and $y$ have the same volume type so we conclude $a_d=b_d$. An argument considering $\det(x^{1}, \cdots, x^{n-1}, x^{n+1},\cdots,x^d,x^i)$ similarly shows that $a_n=b_n$ for every $1\leq n\leq d$.  Thus,

\begin{equation*}
gx^{i}= g\sum_{n=1}^{d}a_nx^{n}=\sum_{n=1}^{d}a_ngx^{n}=\sum_{n=1}^{d}a_ny^{n}=y^i.
\end{equation*}

Note that $g\in\text{SL}_d(\R)$, since $\det g= \det((y^{1},\cdots,y^{d})(x^{1},\cdots,x^{d})^{-1})=\det(y^{1},\cdots,y^{d})\det(x^{1},\cdots,x^{d})^{-1}=1$.  This proves existence. Uniqueness follows from the fact that the configuration contains a basis, so $g$ is determined by its action on the configuration. The converse follows from the matrix equation 
\begin{equation*}
    g(x^{1},\cdots,x^{d})= (y^{1},\cdots,y^{d})
\end{equation*}
and the fact that $g$ has determinant 1. 

\end{proof}

We conclude this section by proving Theorem \ref{thm: manifold}.  Given manifolds $M$ and $N$, a smooth map $\Phi:M\to N$ is an immersion if the derivative $D\Phi$ has full rank everywhere.  A smooth embedding is an injective immersion which is also a topological embedding, i.e. a homeomorphism from $M$ to $\Phi(M)$.  A thorough treatment can be found in chapter 5 of \cite{Lee}.  In particular, we will use the following theorem.

\begin{thm}[\cite{Lee}, Theorem 5.31]
\label{thm: EmbeddedSubmanifolds}
The image of a smooth embedding is an embedded submanifold.
\end{thm}

\begin{proof}[Proof of Theorem \ref{thm: manifold}]
Let $M$ be the subset of $(\R^d)^k$ consisting of configurations of the form

\[
(e^1,...,e^{d-1},te^d,z^{d+1},...,z^k)
\]
with $t\in\R\setminus \{0\},z^i\in\R^d$, where $e^i$ is the $i$-th standard basis vector in $\R^d$.  We claim $M$ has a unique representative of every non-degenerate volume type.  To prove every volume type is represented, let $x\in (\R^d)^k$ be non-degenerate.  Let $t=\det(x^1,...,x^d)$ and let $g\in\text{SL}_d(\R)$ be such that $g(x^1,...,x^d)=(e^1,...,te^d)$.  For $i>d$, let $z^i=gx^i$.  This choice of $t$ and $z^i$ produces an element of $M$ with the same volume type as $x$.  To show this representation is unique, suppose $(e^1,...,e^{d-1},te^d,z^{d+1},...,z^k)$ and $(e^1,...,e^{d-1},t'e^d,w^{d+1},...,w^k)$ have the same volume type.  Considering the volumes of the first $d$ points, it is easy to see $t=t'$.  If $g$ is the element of $\text{SL}_d(\R)$ mapping the first configuration to the second, it follows that $g$ fixes a basis and is therefore the identity. \\

$M$ is a manifold of dimension $d(k-d)+1$, and we can take $t,z^{d+1},...,z^k$ as the coordinates of the point $(e^1,...,e^{d-1},te^d,z^{d+1},...,z^k)$.  If $\Phi(t,z^{d+1},...,z^k)$ is the volume type of $(e^1,...,e^{d-1},te^d,z^{d+1},...,z^k)$, then we have a smooth injective map $\Phi:M\to\R^{\binom{k}{d}}$.  We have

\[
t=\det(e^1,...,te^d),\ \ \ \ \ \text{and}\ \ \ \ \ tz_j^i=\det(e^1,...,e^{j-1},z^i,e^{j+1},...,te^d).
\]

Let $R_0$ be the row of the matrix $D\Phi$ corresponding to the component $\det(e^1,...,te^d)$, and for each $i,j>d$ let $R_{i,j}$ be the row corresponding to the component $\det(e^1,...,e^{j-1},z^i,e^{j+1},...,te^d)$.  Then $R_0$ has a 1 in the column corresponding to $\partial/\partial t$ and $0$ elsewhere.  The row $R_{i,j}$ has a $t$ in the column corresponding to $\partial/\partial z_j^i$, a $z_j^i$ in the column corresponding to $\partial/\partial t$, and $0$ elsewhere.  It is therefore clear that $D\Phi$ has full rank, so $\Phi$ is an immersion.  It is also clear that $\Phi$ and $\Phi^{-1}$ are smooth, so $\Phi$ is an embedding.  It follows from Theorem \ref{thm: EmbeddedSubmanifolds} that the image $\mathcal{V}_{k,d}$ is an embedded submanifold of $\R^{\binom{k}{d}}$.  The dimension of $\mathcal{V}_{k,d}$ must be $\dim M=d(k-d)+1$.

\end{proof}
\section{Bounds on relevant operators}

\subsection{Fourier integral operators and generalized Radon transforms}

To prove our theorems, we will employ the usual strategy of defining pushforward measures supported on our sets $\mathcal{V}_{k,d}(E)$ and $\mathcal{C}_{k,d}(E)$, taking approximations to those measures, and obtaining a uniform $L^2$ bound on those approximations.  This will reduce to using mapping properties of generalized Radon transforms, which we establish here.  We will be following the framework introduced in \cite{GIT20}. \\

Let $X$ and $Y$ be open subsets of $\R^{d\times(d-1)}$ and $\R^d$, respectively.  A \textbf{symbol} of order $m$ on $X\times Y\times \R$ is a smooth map $a:X\times Y\times \R\to\R$ satisfying the bound

\[
\left|\frac{\partial^n}{\partial \theta^n}a(x,y,\theta)\right|\lesssim (1+|\theta|)^{m-n}
\]

on compact subsets of $X\times Y$.  Also, for smooth phase functions $\f:X\times Y\times \R\to \R$, define

\[
C_\f=\left\{(x,\nabla_x \f(x,y,\theta),y,-\nabla_y \f(x,y,\theta):\theta\neq 0,\frac{\partial}{\partial\theta}\f(x,y,\theta)= 0\right\}.
\]

We view $C_\f$ as a subset of $(T^*X\setminus \{0\})\times (T^*Y\setminus\{0\})$.  Given any subset $C\subset (T^*X\setminus \{0\})\times (T^*Y\setminus\{0\})$ and any order $m\in\R$, define the class of Fourier integral operators of order $m$ and with canonical relation $C$, denoted by $I^m(C)$, to be those with Schwartz kernels which are locally finite sums of kernels of the form

\[
K(x,y)=\int e^{i\f(x,y,\theta)}a(x,y,\theta)\:d\theta
\]

where $C_\f$ is a relatively open subset of $C$ and $a$ is a symbol of order $m-\frac{1}{2}+\frac{d^2}{4}$.  We will use the following result.

\begin{thm}[\cite{GIT20}, Theorem 3.1]
\label{thm: FIOBound}
Let $C$ be a canonical relation and let $A\in I^{r-\frac{d^2-2d}{4}}$ have compactly supported Schwartz kernel.  Suppose the projections from $(T^*X\setminus \{0\})\times (T^*Y\setminus\{0\})$ to each factor, restricted to $C$, have full rank (so the first is an immersion and the second is a submersion).  Then $A$ is a bounded operator $L^2(Y)\to L_{-r}^2(X)$.
\end{thm}

Let $\Phi,\eta:X\times Y\to \R$ be smooth and let $\eta$ be compactly supported.  A \textbf{generalized Radon transform} is an operator of the form

\[
Af(x)=\int_{\Phi(x,y)=0}f(y)\eta(x,y)\:d\sigma_x(y),
\]

where $\sigma_x$ is the induced surface measure on the surface defined by $\Phi(x,y)=0$.  This can be written in terms of the delta distribution (and its Fourier transform) as an oscillatory integral; we have

\begin{align*}
    Af(x)&=\int_{\Phi(x,y)=0}f(y)\eta(x,y)\:d\sigma_x(y) \\
    &=\int \delta(\Phi(x,y))f(y)\eta(x,y)\:dy \\
    &=\int\int e^{2\pi i \Phi(x,y)\theta}f(y)\eta(x,y)1(\theta)\:d\theta\:dy
\end{align*}

Therefore, $A$ is a Fourier integral operator with phase function $2\pi\Phi(x,y)\theta$ and amplitude $\eta(x,y)\theta$.  The symbol $\eta(x,y)\theta$ has order 0, so our generalized radon transforms are Fourier integral operators of order $\frac{2-d^2}{4}$.  This means Theorem \ref{thm: FIOBound} applies with $r=-\frac{d-1}{2}$, assuming the condition on the canonical relation holds. \\

The generalized radon transforms we will be interested in are those given by the determinant function.  Throughout this paper, $\mathcal{R}_t$ will denote the operator

\[
\mathcal{R}_tf(x^1,\cdots x^{d-1})=\int_{\det(x^1,\cdots,x^d)=t}f(x^d)\eta(x^1,\cdots,x^d)\:d\sigma_{t,x^1,\cdots,x^{d-1}}(x^d)
\]

where $\sigma_{t,x^1,\cdots,x^{d-1}}$ is the surface measure.  These operators are shown to satisfy the canonical relation hypothesis of Theorem \ref{thm: FIOBound} in \cite{GIT20}, which implies the following Sobolev bound for $\mathcal{R}_t$.

\begin{thm}
\label{thm: SobolevBound}
The generalized Radon transform $\mathcal{R}_t$ defined above is a bounded operator $L^2(\R^d)\to L_{\frac{d-1}{2}}^2((\R^d)^{d-1})$.
\end{thm}

\subsection{Frostman measures and Littlewood-Paley projections}

The following theorem is frequently used to study the dimension of fractal sets; see, for example, \cite{W03}.

\begin{thm}[Frostman's Lemma]
\label{thm: Frostman}
Let $E\subset\R^d$ be compact.  For any $s<\dim E$, there is a Borel probability measure $\mu$ supported on $E$ satisfying

\[
\mu(B_r(x))\lesssim r^s
\]

for all $x\in\R^d$ and all $r>0$.
\end{thm}

A measure $\mu$ as in the theorem is called a Frostman probability measure of exponent $s$. \\

We will be interested in the Littlewood-Paley decomposition of Frostman measures.  Let $\mu$ be a Frostman probability measure on $\R^d$ with exponent $s$ and compact support.  Then $\mu_j$ is the $j$-th Littlewood-Paley piece of $\mu$, defined by $\widehat{\mu_j}(\x)=\psi(2^{-j}\x)\widehat{\mu}(\x)$ where $\psi$ is a Schwarz function supported in the range $\frac{1}{2}\leq |\x|\leq 4$ and constantly equal to 1 in the range $1\leq |\x|\leq 2$.  We will use the following bounds.

\begin{lem}
\label{lem: LittlewoodPaleyBounds}
Let $\mu$ be a compactly supported Frostman probability measure with exponent $s$, and let $(f\mu)_j$ be the $j$-th Littlewood Paley piece of the measure $f\mu$ for a function $f$.  Then

\[
    \|(f\mu)_j\|_{L^\infty}\lesssim 2^{j(d-s)} \|f\|_{L^\infty(\mu)}
    \]
    and
    \[
    \|(f\mu)_j\|_{L^2}^2\lesssim 2^{j(d-s)} \|f\|_{L^2(\mu)}^2
\]
\end{lem}

\begin{proof}

Firstly, let us prove the $L^{\infty}$ bound.  Since $\|(f\mu)_j\|_{L^\infty}\leq \|f\|_{L^\infty(\mu)} \|\mu_j\|_{L^\infty}$ it suffices to prove the bound in the case $f=1$.  Observe that 

\[
    (f\mu)_j(x)=2^{dj}\widecheck{\psi}(2^j\cdot)*f\mu(x)
\]
    
 Since $\psi$ is a Schwarz function, we have $\psi(x)\lesssim (1+|x|)^{-2}$. Therefore,
 
\begin{equation*}
    |\mu_j(x)|\lesssim 2^{dj}\int(1+2^j|x-y|)^{-2} d\mu(y)
\end{equation*}

Splitting this integral into two parts: $2^j|x-y|<1$ and $2^j|x-y|>1$. We have

\begin{align*}
   & 2^{dj}\int_{2^j|x-y|<1}(1+2^j|x-y|)^{-2} d\mu(y)\\
    &\lesssim 2^{dj}\mu(\{y:2^j|x-y|<1\}) \\
    &\lesssim  2^{j(d-s)}
\end{align*}

and 

\begin{align*}
   & 2^{dj}\int_{2^j|x-y|>1}(1+2^j|x-y|)^{-2} d\mu(y)\\
    &= 2^{dj}\displaystyle\sum_{i=0}^{\infty}\int_{2^i\leq 2^j|x-y|\leq 2^{i+1}}(1+2^j|x-y|)^{-2} d\mu(y) \\
    &\lesssim  2^{dj}\displaystyle\sum_{i=0}^{\infty}2^{-2i}\mu(\{y:2^i\leq 2^{j}|x-y|\leq 2^{i+1}\})\\
    &\lesssim  2^{j(d-s)}\displaystyle\sum_{i=0}^{\infty}2^{i(s-2)}\\
    &\lesssim  2^{j(d-s)}
\end{align*}

Thus, we get the first result as claimed.  To prove the $L^2$ bound, we first observe that

\begin{align*}
    \|(f\mu)_j\|_{L^2}^2&=\|\widehat{(f\mu)_j}\|_{L^2}^2 \\
    &=\int|\widehat{f\mu}(\xi)|^2\psi_j^2(\xi)d\xi \\
    &=2^{jd}\int\int \widehat{\psi^2}(2^j(x-y))f(x)f(y) d\mu(x)d\mu(y)
\end{align*}

where we have used Fourier inversion in the last line.  Break the integral into two parts corresponding to $|x-y|< 2^{-j}$ and $|x-y|> 2^{-j}$, where $C$ is a large constant.  Since $\psi$ is a Schwartz function, it suffices to bound the first part.  Let $K_j=2^{dj}\chi_{\{|x-y|<2^{-j}\}}$ and let $T_jf(x)=\int K_j(x,y)f(y)d\mu(y)$.  Our goal is to prove $\left\langle T_jf,f\right\rangle_{L^2(\mu)}\lesssim 2^{j(d-s)}\|f\|_{L^2(\mu)}$.  By Cauchy-Schwarz, it suffices to show the norm of $T_j$ as an operator $L^2(\mu)\to L^2(\mu)$ is bounded by $2^{j(d-s)}$.  This follows from Schur's test, as 

\[
\int K(x,y)d\mu(x)=\int K(x,y)d\mu(y)\lesssim 2^{j(d-s)}.
\]

\end{proof}

The generalized Radon transform applied to $\mu_j$ also has Fourier transform concentrated at scale $2^j$.  This together with Theorem \ref{thm: SobolevBound} allows us to prove the following bounds.  Here and throughout, given $f_1,...,f_n:X\to\R$, the function $f_1\otimes\cdots\otimes f_n$ is the function $X^n\to\R$ given by

\[
f_1\otimes\cdots\otimes f_n(x^1,...,x^n)=f_1(x^1)\cdots f_n(x^n)
\]

\begin{lem}
\label{lem: GRTBounds}
Let $\f$ be a smooth function which is supported on $[-1,1]$ and equal to 1 on $[-1/2,1/2]$, and let $\f^\e(t)=\e^{-1}\f(\e^{-1}t)$.  Let $\eta:(\R^d)^d\to \R$ be a smooth cutoff function supported in the region $|x^i-e^i|<c$ where $e^i$ is the $i$-th standard basis vector and $c$ is a small positive constant.  Finally, let $\mathcal{R}_t^\e$ be the approximate generalized Radon transform defined by
\[
\mathcal{R}_t^\e f(x^1,...,x^{d-1})=\int f(x^d)\eta(x^1,...,x^d)\f^\e(\det(x^1,...,x^d)-t)\:dx^d.
\]

If $c$ is sufficiently small, we have the following.

\begin{enumerate}[(i)]
\item

\[
\|\mathcal{R}_t^\e (f\mu)_j\|_{L^2}^2\lesssim 2^{j(1-s)} \|f\|_{L^2(\mu)}^2.
\]

\item If $j,j_1,...,j_{d-1}$ are any indices such that $|j-j_i|>5$ for any $i$, then for every number $N$ and functions $f,f_1,...,f_{d-1}$ we have

\[
\left\langle \mathcal{R}_t^\e (f\mu)_j,(f_1\mu)_{j_1}\otimes\cdots\otimes (f_{d-1}\mu)_{j_{d-1}}\right\rangle \lesssim_N 2^{-N\cdot\max(j,j_1,...,j_{d-1})},
\]

where $\left\langle \cdot,\cdot \right\rangle$ is the inner product on $L^2(\R^{d-1})$.
\end{enumerate}

\end{lem}

\begin{proof}
We first prove that the Fourier transform of $\mathcal{R}_t^\e(f\mu)_j$ decays rapidly outside the region $|x^j|\approx 2^j$.  After we prove this, both statements follow from Plancherel and Theorem \ref{thm: SobolevBound}.  By Fourier inversion, we have

\[
\mathcal{R}_t^\e\mu_j(x^1,...,x^{d-1})=\int\int\int e^{2\pi i \xi^d\cdot x^d}e^{2\pi i \tau(\det(x)-t)}\widehat{(f\mu)_j}(x^d) \widehat{\f}(\e\tau)\eta(x) \:dx^d\:d\xi^d\:d\tau,
\]

and therefore

\[
\widehat{\mathcal{R}_t^\e\mu_j}(\xi^1,...,\xi^{d-1})= \int\int\int e^{2\pi i (\widetilde{\xi}\cdot x+\tau(\det(x)-t))} \widehat{(f\mu)_j}(\xi^d) \widehat{\f}(\e\tau)\eta(x) \:dx \:d\xi^d \:d\tau
\]

where $\widetilde{\xi}=(\xi^1,...,\xi^{d-1},-\xi^d)$.  This integral can be written

\[
\int\int \widehat{(f\mu)_j}(\xi^d) \widehat{\f}(\e\tau)I(\tau,\xi)\: d\tau \: d\xi^d,
\]

where

\[
I(\tau,\xi)=\int e^{2\pi i (\widetilde{\xi}\cdot x+\tau(\det(x)-t))} \eta(x)\:dx.
\]

This is an oscillatory integral with phase function

\[
\Phi_{\tau,\xi}(x)=\widetilde{\xi}\cdot x+\tau(\det(x)-t).
\]

We observe

\[
\nabla\Phi_{\tau,\xi}(x)=\widetilde{\xi}+\tau\cdot\nabla\det(x).
\]

For $x$ in the support of $\eta$, we have $\frac{1}{2}< |\nabla_{x^i}\det(x)-e^i| < 2$ if the constant $c$ in the statement of the theorem is sufficiently small.  Therefore, if $\Phi_{\tau,\xi}$ has critical points then we must have $\frac{1}{2}|\xi^i| \leq \tau\leq 2|\xi^i|$ for all $i$.  If $2^{j-2}<|\xi^d|<2^{j+2}$ and $2^{j_i-2}<|\xi^i|<2^{j_i+2}$ with $|j-j_i|>5$, then $\Phi_{\tau,\xi}$ has no critical points and by non-stationary phase (for example \cite{W03}, proposition 6.1) we have

\[
I(\tau,\xi)\lesssim_N 2^{-N\cdot \max(j,j_1,...,j_{d-1})}.
\]

It follows from this and Lemma \ref{lem: LittlewoodPaleyBounds} that

\begin{align*}
    \left\langle \mathcal{R}_t^\e (f\mu)_j,(f_1\mu)_{j_1}\otimes\cdots\otimes (f_{d-1}\mu)_{j_{d-1}}\right\rangle &= \left\langle \widehat{\mathcal{R}_t^\e (f\mu)_j},\widehat{(f_1\mu)_{j_1}}\otimes\cdots\otimes \widehat{(f_{d-1}\mu)_{j_{d-1}}}\right\rangle \\
    &=\int\int\widehat{(f_1\mu)_{j_1}}(\xi^1)\cdots\widehat{(f_{d-1}\mu)_{j_{d-1}}}(\xi^{d-1})\widehat{(f\mu)_j}(\xi^d)\widehat{\f}(\e\tau)I(\tau,\xi)\:d\tau\:d\xi \\
    &\lesssim_N 2^{-N\cdot\max(j,j_1,...,j_{d-1})}.
\end{align*}

It also follows that

\begin{align*}
    \|\mathcal{R}_t^\e(f\mu)_j\|_{L^2}^2&= \| \widehat{\mathcal{R}_t^\e(f\mu)_j}\|_{L^2}^2 \\
    &\lesssim 2^{-j(d-1)}\int_{|\xi| \approx 2^j} |\xi|^{d-1} \widehat{\mathcal{R}_t^\e(f\mu)_j}(\xi)d\xi \\
    &=2^{-j(d-1)}\|\mathcal{R}_t^\e(f\mu)_j\|_{L_{\frac{d-1}{2}}^2}^2 \\
    &\lesssim 2^{j(1-s)}\|f\|_{L^2(\mu)}^2
\end{align*}

\end{proof}

\section{Proofs}

\subsection{Proof of Theorem \ref{thm: MT1}}
Many Falconer type problems can be attacked by defining an appropriate pushforward measure and proving it is in $L^2$.  The following lemma establishes this framework.

\begin{lem}
\label{lem: ApproximateNuL2Bound}
Let $\mathcal{M}$ be an $n$-dimensional submanifold of $\R^m$ equipped with $n$-dimensional Lebesgue measure $\mathcal{L}_n$ and consider a map $\Phi:(\R^d)^k\to \mathcal{M}$.  For $E\subset \R^d$, let

\[
\Delta_\Phi(E)=\{\Phi(x):x\in E^k\}.
\]

If $\mu$ is a probability measure supported on a compact set $E$ and

\[
\e^{-n}\int\cdots\int_{|\Phi(x)-\Phi(y)|\lesssim \e}d\mu^k(x)\:d\mu^k(y)\lesssim 1,
\]

then $\mathcal{L}_n(\Delta_\Phi(E))>0$.
\end{lem}

\begin{proof}

Define a probability measure $\nu$ on $\mathcal{M}$ by the relation

\[
\int f(t)\:d\nu(t)=\int f(\Phi(x))\:d\mu^k(x).
\]

It suffices to prove $\nu$ is absolutely continuous with respect to $\mathcal{L}_n$.  Let $\f$ be a symmetric Schwartz function on $\R^m$ supported on the ball of radius $2$ and equal to $1$ on the unit ball.  Let $\f^\e(x)= \e^{-n}\f(x/\e)$ and let $\nu^\e=\f^\e*\nu$.  Then

\[
\int_A\nu^\e(t)dt\leq \mathcal{L}_n(A)^{1/2} \|\nu^\e\|_{L^2},
\]

where $dt$ denotes integration with respect to $n$-dimensional Lebesgue measure.  This reduces matters to proving an upper bound on $\|\nu^\e\|_{L^2}$ which is uniform in $\e$.  We have 
\begin{align*}
\nu^\e(t)&= \int \f^\e (t'-t) d\nu(t')\\
&=\int \f^\e (\Phi(x)-t) d\mu^k(x)\\
&\approx \e^{-n}\int_{|\Phi(x)-t|\leq \varepsilon/2}d\mu^k(x).
\end{align*}

Thus,

\begin{align*}
||\nu^\e||^2_{L^2}&\approx \e^{-2n}\int \left(\int\cdots\int_{|\Phi(x)-t|\leq \varepsilon/2 } d\mu^k(x) d\mu^k(y)   \right)dt\\
&=\e^{-2n}\int \cdots\int_{|\Phi(x)-\Phi(y)|\leq \varepsilon}\left(\int_{|\phi(x)-t|\leq \varepsilon/2 } dt   \right)d\mu^k(x) d\mu^k(y)\\
&\approx \e^{-n}\int \cdots\int_{|\phi(x)-\phi(y)|\leq \varepsilon}d\mu^k(x) d\mu^k(y) \\
&\lesssim 1
\end{align*}

\end{proof}

%%%%%%%%%%%%%%%%%%%%%%%%%%% Working %%%%%%%%%%%%%%%%%%%%%%%

To apply this approach to our current problem, we first reduce to the case where our set $E\subset \R^d$ has some additional structure.

\begin{lem}
\label{lem: Paring}
Let $k \geq d$ and let $E\subset\R^d$ be a compact set with Hausdorff dimension $\dim E>d-1$.  Then there exist subsets $E_1,...,E_k\subset E$ and a constant $c$ with $\dim E_j=\dim E$ and the property that for any choice of $d$ points $x^1,...,x^d$ in different cells $E_j$, we have $\det(x^1,\cdots,x^d)>c$.
\end{lem}

\begin{proof}
Let $\mu$ be a Frostman probability measure on $E$ with exponent $s>d-1$ and let $N$ be a large integer to be determined later.  The idea of the proof is that the $2^{-N}$-neighborhood of a compact piece of a hyperplane has negligible $\mu$-measure, so we can construct our sets $E_j$ recursively by throwing away bad parts of $E$. \\

Given a point $x\in\R^d$, let $B(x)$ denote the ball of radius $2^{-N}$ centered at $x$.  Let $\mathcal{S}_0$ be a finite set such that $\{B(x):x\in\mathcal{S}_0$ covers $E$, and let $\mathcal{S}\subset\mathcal{S}_0$ be the subset obtained by discarding any $x$ such that $\mu(B(x))=0$.  Without loss of generality we may assume that none of our balls contains the origin. \\

Let $x^1,x^2\in\mathcal{S}$ be arbitrary points such that the balls $B(x^1)$ and $B(x^2)$ have distance $>2^{-N}$.  For $2\leq j\leq d-1$, suppose $x^1,...,x^j$ have been defined and are linearly independent.  Let $X$ denote the $2^{-N+10}$-neighborhood of $\text{span}(x^1,...,x^j)$ intersected with the ball of radius $\sup E$.  Then $\mu(X)\lesssim 2^{-N(s-j)}$.  Since $s>j$, for large $N$ this is small, so we can choose $x^{j+1}\in E\setminus X$.  It follows that $B(x^{j+1})$ does not intersect the $2^{-N}$ neighborhood of $\text{span}(x^1,...,x^j)$.  For $d\leq j< k$, suppose $x^1,...,x^j$ have been defined and have the property that for any $j_1,...,j_d\leq j$, $B(x^j)$ does not intersect the $2^{-N}$-neighborhood of $\text{span}(x^{j_1},...,x^{j_{d-1}})$.  Again, if $N$ is sufficiently large then the union of all $\binom{j}{d-1}$ approximate hyplerplanes determined by any $d-1$ of the points $x^1,...,x^j$ has small $\mu$ measure, so we can choose $x^{j+1}$ to avoid all of them as well.  It is clear that the collection $E_j:=B(x^j)$ has the desired properties.

\end{proof}

To prove Theorem \ref{thm: MT1}, by Lemmas \ref{lem: ApproximateNuL2Bound} and \ref{lem: Paring} it suffices to bound

\[
\tag{1}
\e^{-d(k-d)-1}\int\int_{|\Phi(x)-\Phi(y)|\lesssim \e}\:d\mu^k(x)\:d\mu^k(y)
\]

independent of $\e$.  We follow the approach used in \cite{GILP15} and \cite{M20} to reduce matters to the $k=d$ case.  We first decompose the $d\mu^k(y)$ factor into Littlewood-Paley pieces, reducing ($1$) to

\[
\tag{2}
\approx \e^{-d(k-d)-1}\sum_{j_1,...,j_k}\int\int_{|\Phi(x)-\Phi(y)|\lesssim \e}\mu_{j_1}(y^1)\cdots\ \mu_{j_k}(y^k)\:dy^1\cdots dy^k\:d\mu^k(x).
\]

Here $\{\mu_j\}$ are the Littlewood Paley pieces of $\mu$, as defined in Section 3.  Now that we have an integral in $dy$, we want to use the group action framework discussed in Section 2 to turn this into an integral over $\text{SL}_d(\R)$.  The idea is that for fixed $x$, integrating over the region $|\Phi(x)-\Phi(y)|<\e$ is equivalent to integrating over $y\sim gx$ as $g$ varies.  If $\e$ is sufficiently small then $\det(y^1,\cdots,y^d)\neq 0$ for $y$ in this region.  Every such $y$ has the same area type as a configuration of the form

\[
(x_1^1,\cdots,x_{d-1}^d,t_{d^2},\cdots,t_{kd}).
\]

Moreover, there is an open set $U_d\subset\R^{d^2-1}$ such that for every $(g_1^1,\cdots,g_{d-1}^d)\in U_d$ there exists a unique $g\in\text{SL}_d(\R)$ whose matrix has those entries, and the lower right entry is a rational function of the others.  This gives a rational change of variables

\[
y=g(x_1^1,\cdots,x_{d-1}^d,t_{d^2},\cdots,t_{kd}),
\]

where $g$ is viewed in terms of its coordinates.  Since $x$ lives in a fixed compact subset of configuration space, the Jacobian determinant is $\approx 1$ and ($2$) is

\[
\tag{3}
\approx \e^{-d(k-d)-1}\int\int\int_{B_\e}\left(\sum_{j_1,...,j_k} \mu_{j_1}\otimes\cdots\otimes\mu_{j_k}\right)(g(x_1^1,\cdots,x_{d-1}^d,t_{d^2},\cdots,t_{kd}))\:dg\:dt\:d\mu^k(x),
\]

where the two inner integral signs represent integration over the first $d^2-1$ coordinates of $g$ and the $d(k-d)+1$ coordinates $\{t_i\}$, respectively.  The $t_i$ coordinates are integrated over the ball $B_\e$ raidus $\e$ centered at the last $d(k-d)+1$ coordinates of $x$.  Taking the limit as $\e\to 0$, this is

\[
\tag{4}
\sum_{j_1,...,j_k}\int\int\cdots\int \mu_{j_1}(gx^1)\cdots \mu_{j_k}(gx^k)\:d\mu(x^1)\cdots\:d\mu(x^k)\:dg.
\]

Here we make a couple simple reductions.  First, $\mu_j$ is a Schwarz function satisfying the $L^\infty$ bound $\|\mu_j\|_{L^\infty}\lesssim 2^{j(d-s)}$ (see for example \cite{M20}, Lemma 3) which we use to reduce from general $k\geq d$ to the $k=d$ case.  Moreover, the sum over $j_1,...,j_k$ can be reduced to the sum over indices satisfying $j_1\geq \cdots\geq j_k\geq 0$, as negative indices clearly contribute $O(1)$ to the sum and other permutations of indices only change the sum by a multiplicative constant.  Applying the $L^\infty$ bound and running the sum in the indices $j_{d+1},...,j_k$, it follows that ($4$) is

\[
\lesssim \sum_{j_1\geq \cdots\geq j_d}2^{j_d(d-s)(k-d)}\int\int\cdots\int \mu_{j_1}(gx^1)\cdots\mu_{j_d}(gx^d)\:d\mu(x^1)\cdots\:d\mu(x^d)\:dg
\tag{5}.
\]

This reduces matters to the $k=d$ case.  Using the same change of variables in the other direction, this is

\begin{align*}
&\approx \e^{-1}\sum_{j_1\geq \cdots\geq j_d}2^{j_d(d-s)(k-d)} \int\cdots\int_{|\det(x^1,...,x^d)-\det(y^1,...,y^d)|<\e} \mu_{j_1}(y^1)\cdots \mu_{j_d}(y^d)\:dy\:d\mu^k(x)\\
&\approx \e^{-2}\sum_{j_1\geq \cdots\geq j_d}2^{j_d(d-s)(k-d)} \int\int\cdots\int_{\substack{|\det(x^1,\cdots,x^d)-t|<\e \\ |\det(y^1,\cdots,y^d)-t|<\e}} \mu_{j_1}(y^1)\cdots \mu_{j_d}(y^d)\:dy\:d\mu^k(x)\:dt. \\
&\approx \sum_{j_1\geq \cdots\geq j_d}2^{j_d(d-s)(k-d)}
\int\left(\e^{-1}\int_{|\det(x^1,\cdots,x^d)-t|<\e}d\mu^k(x)\right)\left\langle \mathcal{R}_t^\e\mu_{j_1},\mu_{j_2}\otimes\cdots\otimes \mu_{j_{d}}\right\rangle dt
\tag{6}
\end{align*}

where $\left\langle \cdot,\cdot\right\rangle$ denotes the inner product on $L^2((\R^d)^{d-1})$ and $\mathcal{R}_t^\e$ is the approximation to the generalized Radon transform discussed in Section 2.  Let

\[
\nu_{k,d}^\e(t)=\int_{|\det(x^{i_1},\cdots,x^{i_d})-t|<\e}d\mu^k(x).
\]

The quantity in ($1$) we are trying to bound is $\|\nu_{k,d}^\e\|_{L^2}^2$, and we have proved.

\[
\|\nu_{k,d}^\e\|_{L^2}^2\lesssim \sum_{j_1\geq \cdots\geq j_d}2^{j_d(d-s)(k-d)}\int\nu_{d,d}^\e(t)\left\langle \mathcal{R}_t^\e\mu_{j_1},\mu_{j_2}\otimes\cdots\otimes \mu_{j_{d}}\right\rangle \:dt,
\]

Let

\[
S=\sum_{j_1\geq \cdots\geq j_d\geq 0}2^{j_d(d-s)(k-d)}\sup_t\left\langle \mathcal{R}_t^\e\mu_{j_1},\mu_{j_2}\otimes\cdots\otimes \mu_{j_{d}}\right\rangle.
\]

If $S$ is finite, we have

\[
\|\nu_{k,d}^\e\|_{L^2}^2\lesssim \|\nu_{d,d}^\e\|_{L^2}.
\]

Plugging in $k=d$ on the left, we have a uniform bound on $\|\nu_{d,d}^\e\|_{L^2}$ which in turn implies a uniform bound on $\|\nu_{k,d}^\e\|_{L^2}$ for all $k\geq d$.  So, it suffices to prove $S$ is finite under the hypotheses of Theorem \ref{thm: MT1}.  By Lemma \ref{lem: GRTBounds} it is clear that the part of the sum corresponding to indices with $j_d<j_1-5$ converges.  It also follows from Lemma \ref{lem: GRTBounds} and Cauchy-Schwarz that

\[
\sup_t\left\langle \mathcal{R}_t^\e\mu_j,\mu_j\otimes\cdots\otimes \mu_j\right\rangle \lesssim 2^{\frac{j}{2}(1-s+(d-s)(d-1))}.
\]

Therefore,

\[
S\lesssim \sum_{j\geq 0} 2^{j\left((d-s)(k-d)+\frac{1-s+(d-s)(d-1)}{2}\right)}.
\]

The sum will converge if $s>d-\frac{d-1}{2k-d}$, as claimed.
\subsection{Proof of Theorem \ref{thm: MT3}}

To prove Theorem \ref{thm: MT3}, it is enough to establish the following bound.  The theorem then follows from Lemma \ref{lem: ApproximateNuL2Bound}.

\begin{lem}\
\label{lem:chainslemma}
Let $\f^\e$ be an approximation to the identity on $\R$, and let

\[
J_{t,k}^\e=\int\left(\prod_{j=1}^{k+1-d} \f^\e(\det(x^j,...,x^{j+d-1})-t_j)\right) \:d\mu^k(x).
\]

For every $k\geq d$ there is a constant $C_k$ (which does not depend on $t$ or $\e$) such that $J_{t,k}^\e \leq C_k$.

\end{lem}

\begin{proof}
We first prove a bound in the case $k=d$.  Since $J_{t,d}^\e\approx \sum_j\|\mathcal{R}_t^\e\mu_j\|_{L^1(\mu^{d-1})}$, it is enough to prove $\|\mathcal{R}_t^\e\mu_j\|_{L^2(\mu^{d-1})}\lesssim 2^{-cj}$ for some positive $c$.  To accomplish this, fix $t$ and let $T_j^\e f=\mathcal{R}_t^\e(f\mu)_j$.  We want to bound the norm of $T_j^\e$ as an operator $L^2(\mu) \to L^2(\mu^{d-1})$.  To do this, let $g\in L^2(\mu^{d-1})$ be given by $g(x)=g_1(x^1)\cdots g_{d-1}(x^{d-1})$ with $g_i\in L^2(\mu)$.  Using Littlewood-Paley decomposition, Lemma \ref{lem: GRTBounds}, and Cauchy-Schwarz we have

\begin{align*}
    \left\langle T_j^\e f,g \right\rangle_{L^2(\mu^{d-1})} 
    &\lesssim \left\langle \mathcal{R}_t^\e (f\mu)_j,(g_1\mu)_j\otimes \cdots \otimes (g_{d-1}\mu)_j \right\rangle \\
    &\lesssim 2^{\frac{j}{2}(1-s+(d-s)(d-1))}\|f\|_{L^2(\mu)}\|g\|_{L^2(\mu^{d-1})}.
\end{align*}

It follows that the operator norm, and hence $\|\mathcal{R}_t^\e\mu_j\|_{L^1(\mu^{d-1})}$, is bounded by $2^{\frac{j}{2}(1-s+(d-s)(d-1))}$, and this series converges when $s>d-1+\frac{1}{d}$.  This gives the desired bound in the case $k=d$. \\

For $k>d$, let 

\[
\chi_{t,k}^\e(x)=\prod_{j=1}^{k+1-d} \f^\e(\det(x^j,...,x^{j+d-1})-t_j).
\]

We have

\begin{align*}
    J_{t,k}^\e&=\int \chi_{t,k}^\e(x)d\mu^k(x) \\
    &=\int\chi_{\tilde{t},k-1}^\e(\tilde{x})\f^\e( \det(x^{k+1-d},...,x^k)-t_{k+1-d})d\mu^{k-1}(\tilde{x}) d\mu(x^k) \\
    &\approx \sum_j \int\chi_{\tilde{t},k-1}^\e(\tilde{x}) \mathcal{R}_{t_{k+1-d}}^\e\mu_j(x^{k+1-d},...,x^{k-1})d\mu^{k-1}(\tilde{x}) \\
    &\lesssim (J_{\tilde{t},k-1}^\e)^{1/2}\sum_j \|\mathcal{R}_{t_{k+1-d}}\mu_j\|_{L^2(\mu^{k-1})} \\
    &\lesssim (J_{\tilde{t},k-1}^\e)^{1/2}
\end{align*}
\end{proof}

Let $\Phi(x)=(\det(x^1,...,x^d),...,\det(x^{k+1-d},...,x^k))$.  We have

\[
    \e^{-(k+1-d)}\int\int_{|\Phi(x)-\Phi(y)|<\e}d\mu^k(x)d\mu^k(y) \lesssim \int J_{\Phi(x),k}^\e d\mu^k(x) \lesssim 1.
\]

Theorem \ref{thm: MT3} then follows from Lemma \ref{lem: ApproximateNuL2Bound}.

\subsection{Proof of Sharpness Theorem}
We conclude this paper by proving Theorem \ref{thm:sharpness}. Let $\Lambda_{q,s}$ be the $q^{-\frac{d}{s}}$-neighborhood of $\frac{1}{q}\left(\mathbb{Z}^d\bigcap \left([\frac{q}{2},q]\times [0,q]^{d-1}\right)\right)$, the right half of the lattice in the $(d-1)$-dimensional unit cube with spacing $\frac{1}{q}$ (see figure 2). By Theorem 8.15 in \cite{F286} we can choose a sequence $q_n$ that increases sufficiently rapidly such that 
$$\dim \displaystyle\left(\bigcap_{n}\Lambda_{q_n,s}\right)=s$$

Thus, for large $q$ we may regard $\Lambda_{q,s}$ as an approximation to a set of Hausdorff dimension $s$. Let us modify this situation to fit our problem. By Lemma 1.8 in \cite{F286} we have

\begin{lem}[\cite{F286}, Lemma 1.8]
Let $\psi$ be Lipschitz and surjective, and let $\mathcal{H}^s$ be the s-dimensional Hausdorff measure. Then $\mathcal{H}^s(F)\lesssim \mathcal{H}^s(E)$.
\end{lem}

As consequence of this lemma we have $\dim F \leq \dim E$. If $\psi$ is bijective and Lipschitz in both directions, then $\dim F = \dim E$. Let $E_{q,s}$ (figure 3) be the image of $\Lambda_{q,s}$ under the spherical map

$$\psi(x_1,x_2,\cdots,x_d)=x_1\left(\cos\left(\frac{\pi x_2}{2}\right),\sin\left(\frac{\pi x_2}{2}\right)\cos\left(\frac{\pi x_3}{2}\right),\cdots,\prod_{i=2}^{d-1}\sin\left(\frac{\pi x_i}{2}\right)\cos\left(\frac{\pi x_{d}}{2}\right), \prod_{i=2}^{d}\sin\left(\frac{\pi x_i}{2}\right)\right)$$.

 is not hard to check this map is injective on $[\frac{1}{2},1]\times [0,1]^{d-1}$ and therefore bijective as a map $\Lambda_{q,s}\rightarrow E_{q,s}$. Moreover, let us fix a sequence $q_n$ such that $\dim \displaystyle \left(\bigcap_{n} E_{q_n,s}\right)=s$ and call $E_{s}=\bigcap_{n} E_{q_n,s}$. It remains to prove $\mathcal{L}_{d(k-d)+1}\left(\mathcal{V}_{k,d}\left( E_{s}\right)\right)=0.$ \\

\begin{figure}[ht!]
  \centering
  \begin{minipage}[b]{0.4\textwidth}
    \includegraphics[width=\textwidth]{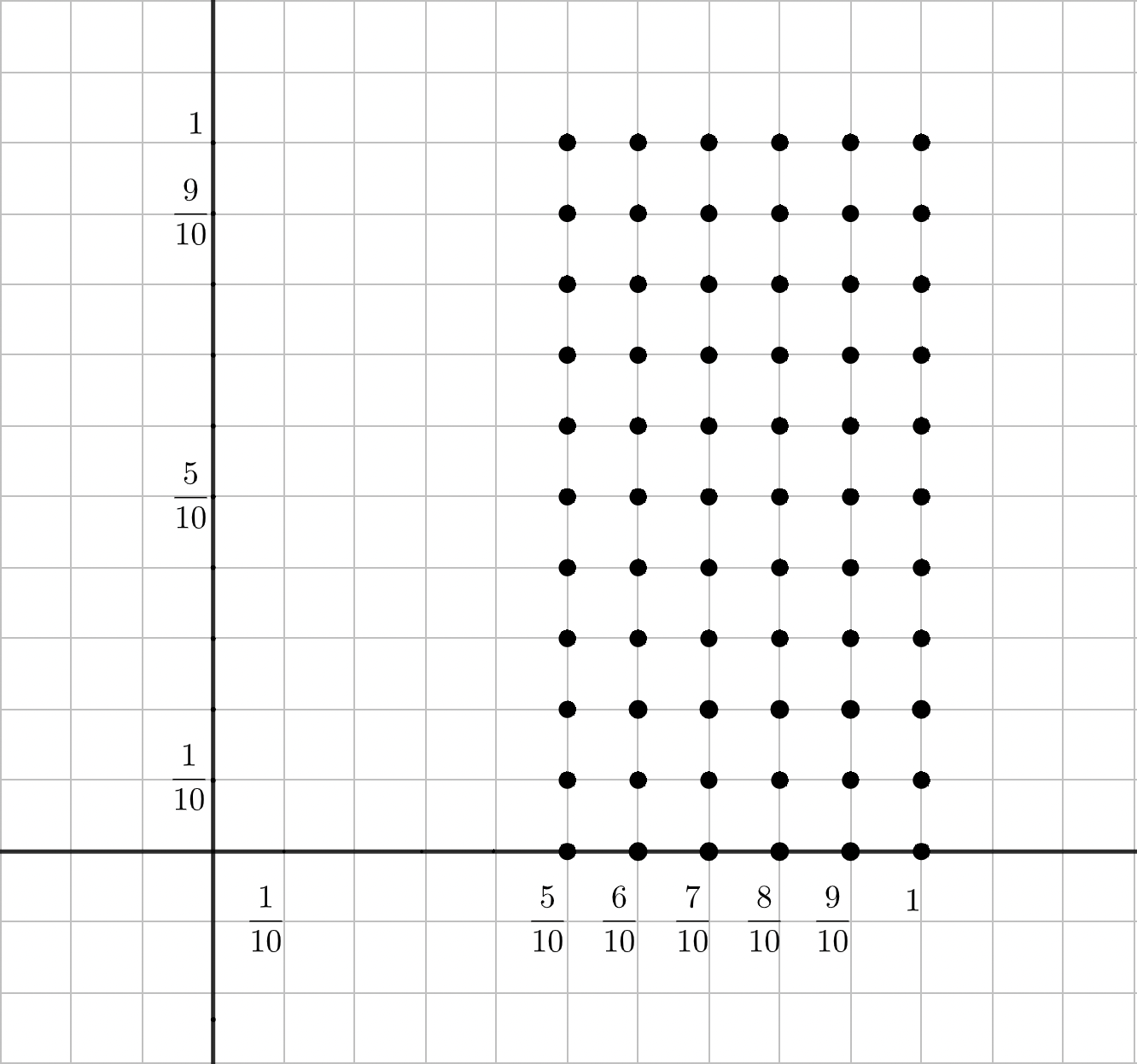}
    \caption{$\Lambda_{10,s}$ for $d=2$.}
  \end{minipage}
  \hfill
  \begin{minipage}[b]{0.4\textwidth}
    \includegraphics[width=\textwidth]{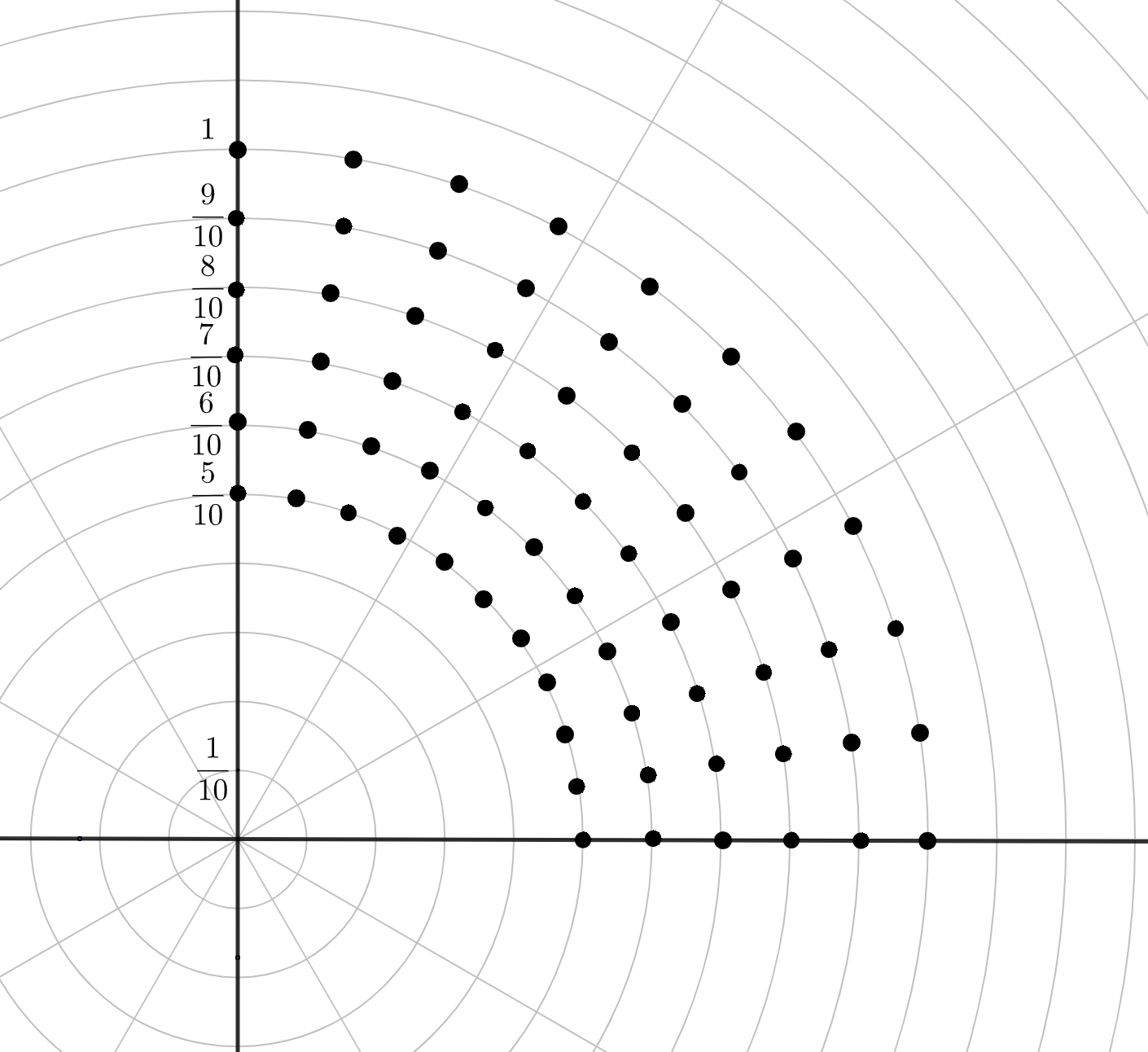}
    \caption{$E_{10,s}$ for $d=2$.}
  \end{minipage}
\end{figure}

% Pictures: q=4 and d=3.  Maybe  ?       %
%\begin{figure}[ht!]
%  \centering
%  \begin{minipage}[b]{0.45\textwidth}
%    \includegraphics[width=\textwidth]{D figure.png}
%    \caption{$\Lambda_{4,s}$ for $d=3$.}
%  \end{minipage}
%  \hfill
%  \begin{minipage}[b]{0.45\textwidth}
%    \includegraphics[width=\textwidth]{E figure.png} 
%    \caption{$E_{4,s}$ for $d=3$.}
%  \end{minipage}
%\end{figure}

We begin by counting the number of volume types determined by the image of $\frac{1}{q}\left(\mathbb{Z}^d\bigcap \left([\frac{q}{2},q]\times [0,q]^{d-1}\right)\right)$ under $\psi$ (i.e., the sperical lattice points themselves and not the thickened set).  It is clear that every volume type of this set is obtained by considering configurations with $x^1$ restrained to the first axis, and $x^2,...,x^k$ unrestrained.  Thus there are $\approx q$ choices for $x^1$ and $\approx q^d$ choices for $x^2,...,x^k$.  It follows that

$$\mathcal{L}_{d(k-d)+1}\left(\mathcal{V}_{k,d}(E_{q,s})\right)\lesssim \left(q^{-\frac{d}{s}}\right)^{d(k-d)+1}q^{d(k-1)+1}$$

This tends to 0 as $q\rightarrow\infty$ provided $s<d-\frac{d^2(d-1)}{d(k-1)+1}$.

\end{document}